\documentclass{article}
\usepackage[utf8]{inputenc}

\usepackage{xr-hyper}

\usepackage{amsthm}
\usepackage{amsmath}
\usepackage{amssymb}
\usepackage[colorlinks,citecolor=blue,urlcolor=blue,filecolor=blue]{hyperref}
\usepackage{graphicx}
\usepackage{bbm}
\usepackage{amssymb}
\usepackage{aligned-overset}
\usepackage{color}
\usepackage{caption}
\usepackage{subcaption}
\usepackage{booktabs}
\usepackage{enumitem,cleveref}
\usepackage[ruled,linesnumbered,noend]{algorithm2e}
\SetKwFor{While}{while}{}{}
\usepackage[noend]{algpseudocode}

\newlist{myenum}{enumerate}{3}
\setlist[myenum,1]{label={\normalfont(\alph*)},
                   ref  ={\normalfont(\alph*)}}
\setlist[myenum,2]{label={\normalfont(\arabic*)},
                   ref  =\themyenumi{.\normalfont{(\arabic*)}}}
\crefname{myenumi}{}{}
\crefname{myenumii}{}{}

\theoremstyle{plain}

\newtheorem{definition}{Definition}

\newtheorem{lemma}{Lemma}

\newtheorem{corollary}{Corollary}

\crefname{assumption}{assumption}{assumptions}
\Crefname{assumption}{Assumption}{Assumptions}

\theoremstyle{remark}
\newtheorem{remark}{Remark}[]

\newcommand{\R}{\mathbb{R}}
\newcommand{\eps}{\varepsilon}

\newcommand{\F}{\mathcal{F}}

\usepackage{todonotes}

\usepackage[natbib,
maxbibnames=4,
sortcites, 
style=alphabetic
]{biblatex}
\title{Please, not \textit{another} note about Generalized Inverses}
\author{Philipp Wacker\thanks{University of Canterbury, Christchurch, New Zealand, \url{phkwacker@gmail.com}}}
\date{May 2023}

\begin{document}

\maketitle

\begin{abstract}
    We prove some statements of left- and right-continuous variants of generalized inverses of non-decreasing real functions. 
\end{abstract}

\section{Introduction}
This short manuscript fills a few theoretical gaps in the recorded knowledge about generalized inverses (also called \textit{quantile functions} in the context of probability theory), and corrects a few inaccuracies in the existing literature. While there is a certain overlap with parts of \cite{vicenik1999note,klement1999quasi,feng2012note,embrechts2013note,de2015study,kampke2015generalized}, none of these give the full picture of generalized inverses, and there are persistent errors that need rectification. Finally, this note\footnote{this is a popular title for communicating results about generalized inverses, see the references section.} presents some (as far as the author is aware) new insights about the exact form of $T\circ T^{-1}$ and $T^{-1}\circ T$, where $T^{-1}$ is a generalized inverse. 

\textbf{Notation:} In the following, we define $f(x+) = \lim_{\eps\searrow x}f(x+\eps)$ and $f(x-) = \lim_{\eps\searrow x}f(x-\eps)$ for a function $f:\R\to\R$. Similarly, $f(-\infty) = \lim_{x\to -\infty}f(x)$ and $f(\infty) = \lim_{x\to +\infty}f(x)$. A non-decreasing function is said to be a map $f:\R\to \R$ such that $x<y$ implies $f(x)\leq f(y)$. We denote by $\overline \R = \R \cup \{-\infty,\infty\}$ the set of the extended real numbers.

We start by defining generalized inverses.
\begin{definition}[generalized inverse]
Let $T:\R\to \R$ be a non-decreasing function where we set $T(-\infty) = \lim_{x\to-\infty}T(x)$ and $T(\infty)=\lim_{x\to\infty}T(x)$. Then the generalized inverses $T^+:\R\to\bar \R$ and $T^-:\R\to\bar \R$ of $T$ are defined by
\begin{align}
    T^+(y) &= \inf\{x\in \R: T(x) > y\}\\
    T^-(y) &= \inf \{x\in \R: T(x) \geq y\}.
\end{align}
with the convention that $\inf \emptyset = \infty$ and $\inf \R = -\infty$.
\end{definition}
\begin{remark}
    \cite{feng2012note} proved that we can equivalently write $T^+(y) = \sup\{x\in\R: T(x) \leq y\}$ and $T^-(y) = \sup\{x\in\R: T(x) <y\}$, as long as we make sure that the domain of $T$ is the whole of $\R$.
\end{remark}
\begin{remark}
    $T^+$ and $T^-$ are the right- and left-continuous generalized inverses of $T$, in the sense outlined by lemma \ref{lem:generalizedinv_prop}\ref{item:cadlag} below.
\end{remark}

\section{Useful statements for working with generalized inverses}
We follow up with a list of elementary properties of $T^+$. Parts \ref{item1}--\ref{item:boundT-T+} are a generalization of \cite[Proposition 1]{embrechts2013note} to the case of both $T^+$ and $T^-$. Part \ref{item:T-unequalT+} is proven in \cite[section 2]{klement1999quasi}. Parts \ref{item:relations} is similar to \cite[Proposition 4.2]{de2015study}, and prove and sharpen all results to cover both the case $T^+$ and $T^-$. \ref{item:relations_rightcont} and \ref{item:relations_leftcont} are new and show what we can say if $T$ is left- or right-continuous. Parts \ref{itemTT+-} and \ref{itemTT+-2} correct a mistake in \cite[Proposition 4.3]{de2015study} (see remark below), and generalize the statement to handle $T^-$, as well. 
\begin{lemma}\label{lem:generalizedinv_prop} Let $T:\R\to\R$ be a nondecreasing map.
    \begin{myenum}    
        \item \label{item1} 
            \begin{myenum}
                \item $T^+(y) = -\infty$ if and only if $T(x) > y$ for all $x\in \R$.
                \item $T^+(y) = \infty$ if and only if $T(x)\leq y$ for all $x\in \R$.
                \item $T^-(y) = -\infty$ if and only if $T(x) \geq y$ for all $x\in \R$.
                \item $T^-(y) = \infty$ if and only if $T(x)< y$ for all $x\in \R$.
            \end{myenum} 
        \item \label{item:nondecreasing}$T^+$ and $T^-$ are nondecreasing.
        \item \label{item:cadlag}$T^+$ is right-continuous, and $T^-$ is left-continuous. For all $y\in\R$,
        \begin{myenum}
            \item $T^+(y-) = T^-(y-) = T^-(y)$
            \item $T^-(y+) = T^+(y+) = T^+(y)$
        \end{myenum}
        \item \label{disc} $T^+$ is continuous at $y$ if and only if $T^-$ is continuous at $y$.
        \item \label{item:boundT-T+}$T^-(y)\leq T^+(y)$
        \item \label{item:T-unequalT+} $T^-(y) = T^+(y)$ if and only if $\mathrm{Card} (T^{-1}(\{y\})\leq 1$.
        \item \label{item:relations}The following relations hold:
        \begin{myenum}
            \item \label{item2dot6} If $y \leq T(x)$, then $T^-(y) \leq x$. Equivalently, \\if $x < T^-(y)$, then $T(x) < y$. 
            \item \label{item2dot5} If $y < T(x)$, then $T^+(y)\leq x$. Equivalently, \\if $x < T^+(y)$, then $T(x) \leq y$.
            \item \label{itemT-T}$T^-(T(x))\leq x$
            \item \label{item2dot65} If $y > T(x)$, then $T^-(y) \geq x$. Equivalently, \\if $x > T^-(y)$, then $T(x)\geq y$.
            \item \label{item2dot75} If $y \geq T(x)$, then $T^+(y) \geq x$. Equivalently, \\if $x > T^+(y)$, then $T(x) > y$. 
            \item \label{itemT+T}$T^+(T(x))\geq x$.
            \item \label{item3} If $T^+(T(x)) = T^-(T(x))$, then $T^+(T(x)) = T^-(T(x)) = x$
            \item \label{item4} $T(T^+(y)-) \leq y$
        \end{myenum}
        \item \label{item:relations_rightcont}Let $T$ be right-continuous at $x$. Then the following relations hold:
        \begin{myenum}
            \item \label{item2dot8}If $y>T(x)$, then $T^-(y) > x$. Equivalently, \\if $x\geq T^-(y)$, then $T(x)\geq y$.
            \item \label{item2dot9}If $y>T(x)$, then $T^+(y) > x$. Equivalently, \\if $x\geq T^+(y)$, then $T(x)\geq y$.
            \item \label{equiv_rightcont} $ y \leq T(x)$ if and only if $T^-(y)\leq x$.
        \end{myenum}
        \item \label{itemTT+-}If $T$ is right-continuous at $x=T^+(y)$, then $T(T^+(y)) \geq y$. If $T$ is right-continuous at $x=T^-(y)$, then $T(T^-(y))\geq y$.
        \item \label{item:relations_leftcont}Let $T$ be left-continuous at $x$. Then the following relations hold:
        \begin{myenum}
            \item \label{itemrlc1}If $y<T(x)$, then $T^-(y) < x$. Equivalently, \\if $x\leq T^-(y)$, then $T(x)\leq y$.
            \item \label{itemrlc2}If $y<T(x)$, then $T^+(y) < x$. Equivalently, \\if $x\leq T^+(y)$, then $T(x)\leq y$.
            \item \label{equiv_rightcont2} $ y \geq T(x)$ if and only if $T^+(y)\geq x$.
        \end{myenum}
        \item \label{itemTT+-2}If $T$ is left-continuous at $x=T^+(y)$, then $T(T^+(y)) \leq y$. If $T$ is left-continuous at $x=T^-(y)$, then $T(T^-(y))\leq y$.        
        \item \label{item:TT+=x} If $T$ is continuous at $T^+(y)$, then $T(T^+(y)) = y$. If $T$ is continuous at $T^-(y)$, then $T(T^-(y)) = y$.
        \item \label{equal2} If $T$ is constant on an interval  $I=(x_1, x_2)$, then for all $x\in I$ we have $T^+(T(x)) > x > T^-(T(x))$.
        \item \label{version} We define the left-continuous and right-continuous versions $T_l(x):=T(x-)$ and $T_r(x):=T(x+)$ of $T$. Then $T_l^+ = T_r^+$ as well as $T_l^-=T_r^-$.
    \end{myenum}
\end{lemma}
\begin{proof}
    \ref{item1}, \ref{item:nondecreasing}, \ref{item:cadlag}, \ref{disc} and \ref{item:boundT-T+} follow immediately from elementary properties of the infimum as well as the monotonicity of $T$. We just prove part of \ref{item:cadlag} for illustration:  Let $A_0 = \{x:T(x)>y\}$ and $A_\eps = \{x:T(x)>y+\eps\}$. Then $A_0 = \bigcup_{\eps>0}A_\eps$ and thus 
    \begin{align*}
        T^+(y) &= \inf A_0 = \inf_{\eps>0}\inf A_\eps = \inf_{\eps>0}T^+(y+\eps)\\
        &= \lim_{\eps\searrow 0}T^+(y+\eps).
    \end{align*}

    Regarding \ref{item:relations}: 
    \begin{enumerate}
            \item[\ref{item2dot6}] Follows directly from definition and the infimum: Let $T(x) \geq y$, then $x\in A:= \{\xi\in\R: T(\xi) \geq y\}$, i.e. $T^-(y) = \inf A \leq x$.
            \item[\ref{item2dot5}] Follows directly from definition and the infimum: Let $T(x) > y$, then $x \in A:= \{\xi\in \R: T(\xi) > y\}$, i.e. $T^+(y) = \inf A \leq x$
            \item[\ref{itemT-T}] This follows from \ref{item2dot6}, by setting $y=T(x)$.
            \item[\ref{item2dot65}] We assume that $y >  T(x)$. Thus for any $\xi\in \R$ with the property that $T(\xi)\geq y$, we have $T(\xi)>T(x)$, i.e. $\xi > x$ by monotonicity of $T$. Since $A\subset B$ implies $\inf A \geq \inf B$, this shows $T^-(y) = \inf \{\xi\in \R: T(\xi) \geq y\} \geq \inf \{\xi\in \R: \xi > x\} = x$.
            \item[\ref{item2dot75}]  We assume that $y\geq T(x)$. Thus for any $\xi\in \R$ with the property that $T(\xi)>y$, we have $T(\xi)>T(x)$, i.e. $\xi > x$ by monotonicity of $T$. Since $A\subset B$ implies $\inf A \geq \inf B$, this shows $T^+(y) = \inf \{\xi\in \R: T(\xi) > y\} \geq \inf \{\xi\in \R: \xi > x\} = x$.
            \item[\ref{itemT+T}] This follows from \ref{item2dot75}, by setting $y=T(x)$
            \item[\ref{item3}] Follows from \ref{itemT-T} and \ref{itemT+T} .
            \item[\ref{item4}] Clearly, $x:=T^+(y)-\eps < T^+(y)$, thus by \ref{item2dot5}, $T(x) = T(T^+(y)-\eps) \leq y$, which shows the statement via $\eps \to 0$.
    \end{enumerate}
     Regarding \ref{item:relations_rightcont}: 
    \begin{enumerate}
            \item[\ref{item2dot8}] We prove the equivalent characterization: Let $x\geq T^-(y)$. We choose a sequence $\xi_n\searrow x$, with $\xi_n > x $, hence $\xi_n > T^-(y)$, and thus $T(\xi_n) \geq y$ (by \ref{item2dot65}). Using right continuity of $T$, we see that $T(x) = \lim_n T(\xi_n)\geq y$.
            \item[\ref{item2dot9}] This follows from \ref{item2dot8} and \ref{item:boundT-T+}
            \item[\ref{equiv_rightcont}] is a direct implication of \ref{item2dot6} and \ref{item2dot8}
    \end{enumerate}
    
    \ref{itemTT+-} follows from \ref{item2dot8} and \ref{item2dot9} by setting $x=T^-(y)$ and $T^+(y)$, respectively.
    
    \ref{item:relations_leftcont} is proven quite similarly to \ref{item:relations_rightcont}.
    
    \ref{itemTT+-2} is proven quite similarly to \ref{itemTT+-}.

    Regarding \ref{item:TT+=x}: Assuming continuity of $T$ at $T^+(y)$ or $T^-(y)$, respectively, the statement follows from an application of \Cref{itemTT+-} and \Cref{itemTT+-2}.
    
    \ref{equal2} is proven as follows. Since $T(x)=y$ for all $x\in (x_1,x_2)$, \ref{item2dot6} and \ref{item2dot75} imply that $T^-(y) \leq x$ for all $x\in I$ as well as $T^+(y) \geq x$ for all $x\in I$, i.e. by taking the limit $x\to x_{1/2}$, we have $T^-(y)\leq x_1$ and $T^+(y) \geq x_2$. Thus, for any $x\in I$, we get $T^+(T(x)) = T^+(y) \geq x_2 > x > x_1 \geq T^-(y) = T^-(T(x))$.

    Lastly we prove \ref{version}. Let $T$ be an arbitrary non-decreasing function, and $T_r$ as above. We fix an arbitray $y$ and set $M_l=\{x:T(x-) > y\}$ and $M_r=\{x: T(x+)>y\}$, and thus $T_l^+(y) = \inf M_l$ and $T_r^+(y) = \inf M_r$. By elementary inclusion $M_l\subseteq M_r$, i.e. $T_l^+(y)\geq T_r^+(y)$. It remains to show the opposite inequality. If $M_l = M_r$, then the statement follows directly. Otherwise, let $x^\star \in M_r\setminus M_l$ which means that $T(x^\star-) \leq y$ and $T(x^\star+)>y$. We will now show that $x^\star$ is a lower bound for $M_r$. Indeed, let $x < x^\star$, then $T(x+)\leq T(x^\star-)\leq y$, i.e. $x\not\in M_r$. Since $M_l\subseteq M_r$, this also shows that $x^\star$ is a lower bound for $M_l$. Now we show that $x^\star \in \overline{M_l}$. Indeed, if we set $\eps>0$, then $T((x^\star + \eps)-)\geq T(x^\star+)> y$, i.e. for all $\eps>0$, $x_\eps := x^\star + \eps \in M_l$. Since $x_\eps \searrow x^\star$,  $x^\star \in \overline{M_l}$. By inclusion, $x^\star \in \overline{M_r}$, as well. $x^\star$ being a lower bound and a limit point proves that $x^\star = \inf M_l$, and $x^\star = \inf M_r$, i.e. $x^\star = T_l^+(y) = T_r^+(y)$. The version of with $T^-$ is proven analogously.

\end{proof}

\begin{remark}
    \cite[Section 3.2]{embrechts2013note} comments on the fact that there are errors in previous work on generalized inverses and constructs a series of four statements in published manuscripts with counterexamples for why they are wrong, but without providing a way of resolving these contradictions. This manuscript does: \ref{itemTT+-}, \ref{itemTT+-2}, and \ref{item:TT+=x} give the exact conditions for what we can say about $T\circ T^+$ as well as $T\circ T^-$. In particular, the counterexamples from \cite{embrechts2013note} do not apply here, since $\infty$ is not a point of (right- or left-)continuity of $T$ (it does not even make sense to think about this).
\end{remark}

There is one particular application that is especially interesting in practice, which is the inverse sampling method for univariate random variables. This is a well-known fact, but it is an elementary direct result of the previous lemma.

\begin{corollary}
    Let $(\Omega,\mathcal B, \mathbb P)$ be a probability space and $X:\Omega\to \R$ a random variable with $F_X$ being its cumulative distribution function. Then the push-forward of a uniform random variable $U([0,1])$ under the generalized inverse $F_X^-$ is the law of $X$. This means that we can generate independent samples $x_i\sim U([0,1])$, plug them into $F_X^-$, and the $\{F_X^-(x_i)\}$ will be samples from $X$.
    \begin{proof}
        Since $F_X$ is right-continuous, we know (from  \Cref{lem:generalizedinv_prop}\ref{equiv_rightcont}) that $y \leq F_X(x)$ if and only if $F_X^-(y) \leq x$. We compute the cumulative distribution function of $F_X^-(U)$:
        \begin{align*}
            \mathbb P(\F_X^-(U)\leq \lambda) &= \mathbb P(U\leq F_X(\lambda)) \\
            &= F_X(\lambda),
        \end{align*}
        since the cumulative distribution function of $U$ is given by $\mathbb P(U\leq r) = r$ (for $r\in [0,1]$). This means that the law of $F_X^-(U)$ is identical to the law of $X$.
    \end{proof}
\end{corollary}

\section{Jumps and Plateaus}

The following two lemmata are an adaptation of \cite[Lemma 2]{vicenik1999note}, and can also be found in \cite[Proposition 4.3]{de2015study}, but because the former is concerned with $T^-$ instead of $T^+$, discusses ``third order terms'' like $T(T^+(T(x))) > T(x)$ instead of ``second order terms'' like $T^+(T(x)) > x$, and does not prove maximality of the half-open intervals involved, and the latter has a typo (and refers to the proof to the former instead of providing a direct proof), we give a proof of this statement for completeness' sake. Additionally, we fix an error in the literature (see remark \ref{rem:fixedmistake} below).

The following statements relate plateaus and jumps of $T$ and $T^\pm$ to one another. For a visualization of these connections, see \cite{de2015study,kampke2015generalized}.

\begin{lemma}\label{main_lemma_cliff} Let $T$ be nondecreasing.
\begin{myenum}
\item \label{Necessity_cliff}If 
\begin{equation}
T^+(y) > T^-(y)
    \end{equation} then for any $x\in (T^-(y),T^+(y)) = I$, both
\begin{align}
    T(x) &=  y \text{ and }\label{first_necessity_cliff}\\
    T^+(T(x)) &> x > T^-(T(x)) \label{second_necessity_cliff}
\end{align}
and there is no greater interval than $I$ of the same type such that \eqref{first_necessity_cliff} holds.

\item  \label{Sufficiency_cliff} Conversely, if either 
\begin{myenum}
\item \label{first_sufficiency_cliff}there is a proper interval $I = (x_1,x_2)$ such that for all $x_0\in I$, $T(x_0) = y$ or
\item \label{second_sufficiency_cliff}for some $x_0$, we have $T^+(T(x_0)) > x_0$, then with $y := T(x_0)$, or 
\item \label{third_sufficiency_cliff}for some $x_0$, we have $T^-(T(x_0)) < x_0$, then with $y := T(x_0)$, 
\end{myenum}
then
    \begin{equation}\label{sufficiency_result_cliff}
        T^+(y) > T^-(y)
    \end{equation}
\item \label{equiv_cliff} For any given $x$, the following two statements are equivalent:
\begin{itemize}
    \item $T\equiv y$ on a proper interval $(x_1,x_2)$.
    \item $T^+(y) > T^-(y)$.
\end{itemize}
\end{myenum}
\end{lemma}
\begin{proof}
    We start by proving \ref{Necessity_cliff}. Let $x\in (T^-(y),T^+(y))$, then $T^-(y)< x < T^+(y)$, i.e. by an application of \Cref{lem:generalizedinv_prop}\ref{item2dot5} and \ref{item2dot65}, $y \leq T(x) \leq y$, which shows that $T$ is indeed constant on the interval $(T^-(y),T^+(y))$. Now we show maximality. If $x > T^+(y)$, then $T(x) > y$ by \Cref{lem:generalizedinv_prop}\ref{item2dot75}. Similarly, if $x < T^-(y)$, then $T(x) < y$ by \ref{item2dot6}. This shows that there is no larger half-open interval $[a,b)$ on which $T\equiv y$. The relation \eqref{second_necessity_cliff} is a direct implication of \eqref{first_necessity_cliff} (which we just proved to be true) and lemma  \Cref{lem:generalizedinv_prop}\ref{equal2}.

    Regarding \ref{Sufficiency_cliff}: We prove that \ref{first_sufficiency_cliff} implies \eqref{sufficiency_result_cliff}. This is a direct consequence of \Cref{lem:generalizedinv_prop}\ref{equal2}:

    Now  \ref{second_sufficiency_cliff} also implies \eqref{sufficiency_result_cliff}: By \Cref{lem:generalizedinv_prop}\ref{itemT-T}, $T^-(T(x_0)) \leq x_0 < T^+(T(x_0))$. Similarly for \ref{third_sufficiency_cliff} (via \Cref{lem:generalizedinv_prop}\ref{itemT+T}).

     \ref{equiv_cliff} follows from a combination from the two other statements.
\end{proof}

\begin{lemma}\label{main_lemma_plateau}Let $T$ be nondecreasing.
\begin{myenum}
\item \label{Necessity_plateau}If 
\begin{equation}
T(x+) > T(x-)
    \end{equation} then for any $y\in (T(x-),T(x+)) = I$, both
\begin{align}
    T^+(y) &=  x = T^-(y)\label{first_necessity_plateau}
\end{align}
and there is no greater interval than $I$ of the same type such that either equality in \eqref{first_necessity_plateau} holds.

\item  \label{Sufficiency_plateau}Conversely, if there is a proper interval $I = (y_1,y_2)$ such that either  
\begin{myenum}
\item \label{first_sufficiency_plateau}for all $y\in I$, $T^+(y) = x$, or
\item \label{second_sufficiency_plateau}for all $y\in I$, $T^-(y) = x$, then
\end{myenum}
    \begin{equation}\label{sufficiency_result}
        T(x+) > T(x-)
    \end{equation}
    \item \label{equiv_plateau} For any given $x$, the following two statements are equivalent:
\begin{itemize}
    \item $T(x+) > T(x-)$.
    \item $T^+\equiv y \equiv T^-$ on a proper interval $(y_1,y_2)$.
\end{itemize}
\end{myenum}
\end{lemma}

\begin{proof}
    We start by proving \ref{Necessity_plateau}. Let $T(x+) > T(x-)$. Then for any $y\in (T(x-),T(x+))$, we have that for any $\eps > 0$, $T(x-\eps) < y < T(x+\eps)$, i.e. (using again the relevant statements in \cref{lem:generalizedinv_prop}) $x-\eps \leq T^-(y)$ and $T^+(y)\leq x+\eps$. By letting $\eps\to 0$, we obtain the statement. Maximality is proven similarly: Take $y > T(x+)$, i.e. there exists $\eps>0$ such that $T(x+\eps) < y$, and thus $T^-(y) \geq x+\eps$, which shows that $y$ is not an element of the set on which $T^-\equiv x$. Since $T^+ \geq T^-$, this also proves that $y$ is not an element of the set on which $T^+\equiv x$. Maximality from below is proven in the same way.

    We now prove \ref{Sufficiency_plateau}, assuming \ref{first_sufficiency_plateau}, i.e. $T^+(y) \equiv x$ on $(y_1,y_2)$. For any $\eps > 0$, $T^+(y) < x+\eps$, i.e. $T(x+\eps) > y$ for all $y\in (y_1,y_2)$. Thus, $T(x+\eps) \geq y_2$ and $T(x+) \geq y_2$. In the same way we prove $T(x-) \leq y_1$: For any $\eps > 0$, $x-\eps < T^-(y)$, i.e. $T(x-\eps) < y$ for all $y\in (y_1,y_2)$. Thus, $T(x-\eps)\leq y_1$ and $T(x-)\leq y_1$. All in all, this proves the statement since $T(x-) \leq y_1 < y_2\leq T(x+)$. The statement \ref{second_sufficiency_plateau} $\Rightarrow$  \eqref{sufficiency_result} is proven in a similar fashion.
   
     \ref{equiv_plateau} follows from a combination from the two other statements.
\end{proof}

\begin{remark}\label{rem:fixedmistake}
     \Cref{main_lemma_plateau}\ref{equiv_plateau} and \Cref{main_lemma_cliff}\ref{equiv_cliff} are a correction of \cite[Proposition 4.3]{de2015study}. In fact, it is not true that $T(T^+(y)) > T(T^+(y)-)$ implies $T(T^+(y)) > y$ or that $T^+(T(x)) > T^+(T(x)-)$ implies $T^+(T(x)) > x$, as Figure \ref{fig:counterex} shows: We set $x=x_2$. Then $T(x)=y$, and $T^+(T(x)) = T^+(y) = x_2 = x$, i.e. the first condition in \cite[Proposition 4.3, 1.)]{de2015study} holds. But, $T(x)=y\in H(T)$, since $y$ is a plateau of $T$. This is in contradiction to the statement of \cite[Proposition 4.3, 1.)]{de2015study}. 
     \begin{figure}
         \centering
         \includegraphics[width=\textwidth]{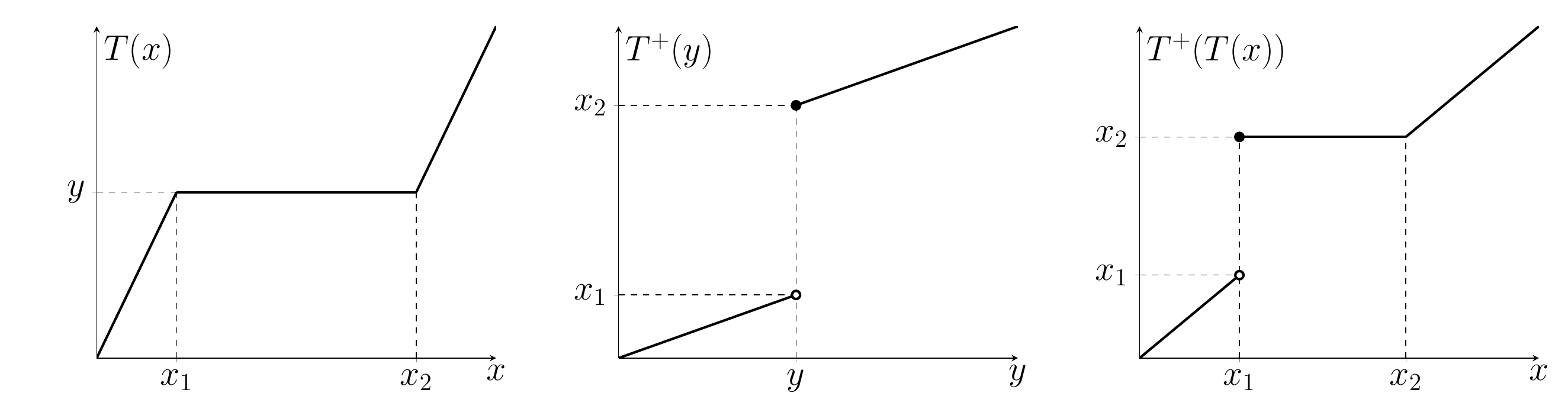}
         \caption{Counterexample: $T^+(T(x)) > T^+(T(x)-)$ but $T^+(T(x)) = x$.}
         \label{fig:counterex}
     \end{figure}
\end{remark}

\section{Inversion Statements}

The remaining statements classify exactly what can be said about $T\circ T^\pm$ and $T^\pm \circ T$ under suitable continuity assumptions.

\begin{lemma}\label{lem:formTT+}
    Let $X = \{x_i\}$ be the (ordered) list of all discontinuities of $T$, where we denote $y_i^+ = T(x_i+)$ and $y_i^- = T(x_i-)$. Then 
    \[T(T^+(y)) = \begin{cases}T(x_i),&\text{ for } y\in(y_i^-,y_i^+)\\
    y,&\text{ for } y \not\in \bigcup_i [y_i^-,y_i^+]
    \end{cases}\]
    \[T(T^-(y)) = \begin{cases}T(x_i),&\text{ for } y\in(y_i^-,y_i^+)\\
    y,&\text{ for } y \not\in \bigcup_i [y_i^-,y_i^+]
    \end{cases}\]
     Let $Y = \{y_i\}$ be the (ordered) list of all discontinuities of $T^\pm$, where we denote $x_i^+ = T^+(y_i)$ and $x_i^- = T^-(y_i)$. Then 
     \[ T^+(T(x)) = \begin{cases}x_i^+=T^+(y_i),&\text{ for } x\in(x_i^-,x_i^+)\\
    x,&\text{ for } x \not\in \bigcup_i [x_i^-,x_i^+]
    \end{cases}\]
    \[
    T^-(T(x)) = \begin{cases}x_i^-=T^-(y_i),&\text{ for } x\in(x_i^-,x_i^+)\\
    x,&\text{ for } x \not\in \bigcup_i [x_i^-,x_i^+]
    \end{cases}   \]
\end{lemma}
\begin{proof}
    We show the characterization for $T\circ T^+$: Let first $y\in (y_i^-,y_i^+)$. Then \Cref{main_lemma_plateau}\Cref{Necessity_plateau} proves $T^+(y) = x_i = T^-(y)$, i.e. $T(T^+(y)) = T(x_i) = T(T^-(y))$. If $y\not \in \bigcup_i [y_i^-,y_i^+]$, then $T^+(y)\not\in X$, or otherwise $T^+(y) = x_j$ for some $j$, which would be in contradiction to the maximality of the set $(y_j^-, y_j^+)$ in  \Cref{main_lemma_plateau}\Cref{Necessity_plateau}. Similarly, $T^-(y)\not\in X$. Thus $T$ is continuous at $T^+(y)$ and at $T^-(y)$, i.e. $T(T^\pm(y)) = y$ by virtue of \Cref{lem:generalizedinv_prop}\Cref{item:TT+=x}.

     Regarding $T^+\circ T$: If $i < j$, $y_i < y_j$ and thus $x_i^+ = T^+(y_i) \leq T^-(y_j) = x_j^-$, so $x_i^+ < x_j^-$ and thus the intervals $(x_i^-,x_i^+)$ are disjoint from another. Let $x\in(x_i^-,x_i^+) = (T^-(y_i),T^+(y_i))$. Then by \Cref{main_lemma_cliff}\ref{Necessity_cliff}, $T^+(T(x))=T^+(y_i) = x_i^+$ and $T^-(T(x)) = T^-(y_i) = x_i^-$. On the other hand, let $x\not\in \bigcup_i [x_i^-,x_i^+]$. Then $T(x)\not\in Y$, because otherwise $T(x)=y_j$ for some $j$, and then $(x_j^-,x_j^+)= (T^+(y_i-),T^+(y_i))$ would  not be the greatest interval possible, in contradiction to \Cref{main_lemma_cliff}\ref{Necessity_cliff}. This means that $T^+$ is continuous at $T(x)$, and thus (because $T^-$ and $T^+$ are left- and right-continuous versions of one another, see \Cref{lem:generalizedinv_prop}\ref{item:cadlag}) $T^+(T(x)) = T^-(T(x))$. By \Cref{lem:generalizedinv_prop}\ref{item3}, $T^+(T(x)) = x = T^-(T(x))$. 
\end{proof}
\begin{remark}
    Note that there is no statement about the edge cases $T(T^\pm(y_i^\pm))$ and $T^\pm(T(x_i^\pm))$. In particular, while $T\circ T^+ = T\circ T^-$ on the two sets considered in the statement of \Cref{lem:formTT+}, it is entirely possible that, e.g., $T(T^+(y_i^+))\neq T(T^-(y_i^+))$. The remaining values depend on the type of continuity of $T$ at those edge points. Assuming global (left- or right-)continuity of $T$ allows us to precisely characterize the invertibility interaction between $T$ and $T^\pm$, and close the gaps in Lemma \ref{lem:formTT+}
\end{remark}

\begin{lemma}\label{lem:TTpmTpmT}
    Let $T$ be nondecreasing and \textbf{continuous from the right}. We denote by $X = \{x_i\}$ the (ordered) list of all discontinuities of $T$, i.e. $y_i^+ := T(x_i)>T(x_i-) =: y_i^-$ and $T(x) = T(x-)$ for $x\not\in X$. We denote by $Y = \{y_i\}$ the (ordered) list of plateaus of $T^\pm$, i.e. for each $y_i$ there exists a proper (maximal in the set of half-open intervals) interval $I_i = [x_i^-, x_i^+)$ such that $T(x)\equiv y_i$ for all $x\in I_i$. Then 
    \begin{align*}
    T(T^+(y)) &= \begin{cases}y_i^+,&\text{ for } y\in[y_i^-,y_i^+)\\
    y,& \text{ else }
    \end{cases}\\
    T^+(T(x)) &= \begin{cases}x_i^+,&\text{ for } x\in[x_i^-,x_i^+)\\
    x,& \text{ else }
    \end{cases}
    \end{align*}
    
    Let $T$ be nondecreasing and \textbf{continuous from the left}. We denote by $X = \{x_i\}$ the (ordered) list of all discontinuities of $T$, i.e. $y_i^+ := T(x_i+)>T(x_i) =: y_i^-$ and $T(x+) = T(x)$ for $x\not\in X$. We denote by $Y = \{y_i\}$ the (ordered) list of plateaus of $T^\pm$, i.e. for each $y_i$ there exists a proper (maximal in the set of half-open intervals) interval $I_i = (x_i^-, x_i^+]$ such that $T(x)\equiv y_i$ for all $x\in I_i$. Then 
    \begin{align*}
    T(T^-(y)) &= \begin{cases}y_i^-,&\text{ for } y\in(y_i^-,y_i^+]\\
    y,& \text{ else }
    \end{cases}\\
    T^-(T(x)) &= \begin{cases}x_i^-,&\text{ for } x\in(x_i^-,x_i^+]\\
    x,& \text{ else }
    \end{cases}
    \end{align*}
    \begin{proof}
        This follows directly from the fact that the concatenation of right-continuous, nondecreasing functions is again right-continuous and non-decreasing (and similarly for left-continuous functions), so we can fill the gaps in our knowledge of, say, $T\circ T^+$ by taking limits from the right, etc. 
    \end{proof}
\end{lemma}

\section{Conclusion}
This manuscript tries to unify, organize, generalize, and correct some statements about generalized inverses. Since this is just the latest work in a long succession of notes claiming to do just that, we close with only cautious optimism of having done so successfully.

\printbibliography

\end{document}